\documentclass[12pt]{elsarticle}
\usepackage[utf8]{inputenc}
\usepackage[english]{babel}
\usepackage{amsmath}
\usepackage{amssymb}
\usepackage{amsfonts}
\usepackage{mathtools}
\usepackage{dsfont}
\usepackage{rotating}
\usepackage{graphicx}
\usepackage{floatflt,epsfig}
\usepackage{lineno,hyperref}
\usepackage{enumerate}
\usepackage{colortbl}
\usepackage{array,tabularx,tabulary,booktabs}
\usepackage{longtable}
\usepackage{multirow}
\usepackage{wrapfig}
\usepackage{subcaption}
\usepackage[table]{xcolor}
\newcolumntype{^}{>{\currentrowstyle}}

\usepackage{amsthm}
\usepackage[left=1in,right=1in,top=1in,bottom=1in]{geometry}

\usepackage{amssymb,amsthm}

\newtheorem{lemma}{Lemma}[section]
\newtheorem{theorem}[lemma]{Theorem}
\newtheorem{corollary}[lemma]{Corollary}

\newtheorem{remark}[lemma]{Remark}

\journal{Designs, Codes and Cryptography}
\begin{document}
\renewcommand{\abstractname}{Abstract}
\renewcommand{\refname}{References}
\renewcommand{\tablename}{Table}
\renewcommand{\arraystretch}{0.9}
\thispagestyle{empty}
\sloppy

\begin{frontmatter}
\title{Divisible design graphs obtained by plugging a difference set into a construction for antipodal distance-regular graphs of diameter 3}

\author[02]{Bart De Bruyn}
\ead{Bart.DeBruyn@UGent.be}

\author[01]{Sergey Goryainov}
\ead{sergey.goryainov3@gmail.com}

\author[03]{Ruilin Ma}
\ead{2868414086@qq.com}

\author[03]{Ruihan Xie}
\ead{3205761192@qq.com}

\address[01] {School of Mathematical Sciences, Hebei International Joint Research Center for Mathematics and Interdisciplinary Science, Hebei Key Laboratory of Computational Mathematics and Applications, Hebei Workstation for Foreign Academicians, Hebei Normal University, Shijiazhuang  050024, P.R. China}
\address[02] {Ghent University, Department of Mathematics, Computer Science and Statistics, Krijgslaan 299 - S9, 9000 Gent, Belgium}
\address[03] {School of Mathematical Sciences, Hebei Normal University, Shijiazhuang 050024, P.R. China}


\begin{abstract}
In this paper, we present a new construction of divisible design graphs with new parameters, obtained by plugging a difference set of a quotient group into a known construction of antipodal distance-regular graphs of diameter 3. Also, we show that in characteristic 2 the new divisible design graphs are Cayley graphs over an elementary abelian 2-group.
\end{abstract}

\begin{keyword}
divisible design graph; distance-regular graph; Cayley graph; difference set
\vspace{\baselineskip}
\MSC[2010] 05B05\sep 05B10\sep 05C25 \sep 05E30 \sep 51A50
\end{keyword}
\end{frontmatter}

\section{Introduction} 
A \emph{divisible design graph} (a DDG for short) with parameters $(v,k,\lambda_1,\lambda_2,m,n)$ is a $k$-regular graph on $v$ vertices such that its vertex set can be partitioned into $m$ classes of size $n$ with the following two properties: any two distinct vertices from the same class have precisely $\lambda_1$ common neighbours and any two vertices from different classes have precisely $\lambda_2$ common neighbours. The partition from the definition of a DDG is called the \emph{canonical partition}. DDGs were introduced in \cite{HKM11,M08} as a bridge between graph theory and design theory (the adjacency matrix of a divisible design graph can be regarded as the incidence matrix of a group divisible design \cite{B77}) and have been studied in \cite{BDG25,BG24,CH14,CS22,DGHS24,GK24,GK25,GHKS19,GHKL25,K22,K23,K26,KS21,MR26,P22,PS22,R26,S21,T22}.

This paper is organised as follows. In Section \ref{sec:prelim}, we give preliminary definitions and results. In Section \ref{sec:construction}, we present a new construction of divisible design graphs such that the number of vertices for each of them is a prime power. In Section \ref{sec:SomeGraphsAreCayley}, we show that the constructed divisible design graphs whose number of vertices is a power of 2 are Cayley graphs over an elementary abelian 2-group. 

\section{Preliminaries}\label{sec:prelim}
In this section we give preliminary definitions and results.

\subsection{Distance-regular graphs}
A \emph{distance-regular graph} is a connected regular graph such that for any two vertices $x$ and $y$, the number of vertices at distance $j$ from $x$ and at distance $\ell$ from $y$ depends only upon $j$, $\ell$, and the distance between $x$ and $y$.
The \emph{intersection array} of a distance-regular graph is the array $(b_{0},b_{1},\ldots ,b_{d-1};c_{1},\ldots ,c_{d})$ in which $d$ is the diameter of the graph and for each $0\leq j\leq d$, $b_{j}$ gives the number of neighbours of $y$  at distance $j + 1$ from $x$, $c_{j}$ gives the number of neighbours of $y$ at distance $j-1$ from $x$ for any pair of vertices $x$ and $y$ at distance $j$. There is also the number $a_j$ that gives the number of neighbours of $y$ at distance $j$ from $x$. The numbers $a_j$, $b_j$, $c_j$ are called the \emph{intersection numbers} of the graph. They satisfy the equation $a_{j}+b_{j}+c_{j}=k$, where $k = b_0$ is the valency, that is, the number of neighbours, of any vertex.

A distance-regular graph of diameter $d$ is called \emph{antipodal} if the relation on its vertex set defined by the rule ``to coincide or to be at distance $d$'' is an equivalence relation; the equivalence classes are then called \emph{antipodal classes}. Note that whether a distance-regular graph is antipodal or not can be told \cite[Proposition 4.2.2(ii)]{BCN89} from its intersection array.

The following theorem gives an infinite family of antipodal distance-regular graphs of diameter 3.

\begin{theorem}[{\cite[p.385, Remark (iv)]{BCN89}}]\label{thm:SymplecticCovers1}
Let $q$ be a prime power. 
Let $V$ be a vector space of arbitrary even positive dimension $2t$ over $\mathbb{F}_q$ provided with a nondegenerate symplectic form $B$. Let $A$ be a subgroup of index $r > 1$ in the additive group of $\mathbb{F}_q^+$. Then the graph $\Gamma(2t,q,A,r)$ with the vertex set $\{(\alpha+A,x) : \alpha \in \mathbb{F}_q, ~x \in V\}$ where $(\alpha+A,x) \sim (\beta+A,y)$ if and only if $B(x,y) \in \alpha - \beta + A$ is an antipodal distance-regular graph with $r\cdot q^{2t}$ vertices and intersection array $\{q^{2t}-1,(r-1)q^{2t}/r,1;1,q^{2t}/r,q^{2t}-1\}$.   
\end{theorem}

\begin{corollary}[{\cite[Proposition 4.13]{HKM11}}] \label{co2.2}
The complements of all the distance-regular graphs from Theorem \ref{thm:SymplecticCovers1} are divisible design graphs.   
\end{corollary}

\subsection{Difference sets}
A $(v,k,\lambda)$ \emph{difference set} is a subset $D$ of size $k$ in a group $G$ of order $v$ such that every non-identity element of $G$ can be expressed as a product $d_1\ast d_2^{-1}$ of elements of $D$ in exactly $\lambda$ ways, that is, when
$$
D\ast D^{-1} = k\cdot\{e\} + \lambda\cdot (G \setminus \{e\})
$$
holds.
Note that, for a $(v,k,\lambda)$ difference set $D$ in a group $G$, the complementary set $G \setminus D$ is a $(v,v-k, v-2k+\lambda)$ difference set.
For any group $G$ of order $v$ and any element $g \in G$,
the set $\{g\}$ is a $(v,1,0)$ difference set and the set $G \setminus \{g\}$ is a $(v,v-1,v-2)$ difference set;
such difference sets are called \emph{trivial}. A difference set $D$ is called \emph{reversible} if $D^{-1} = D$.

Difference sets with parameters $(4u^2, 2u^2\pm u,u^2\pm u)$, where $u$ is a positive integer, are called \emph{Hadamard difference sets}. Such difference sets are closely related to Menon designs and regular Hadamard matrices.
There is a fundamental result \cite[p.366]{BJL99} due to Mann (1965) stating that if $D$ is a non-trivial difference set in a 2-group, then the order of this group is an even power of 2 and this difference set is an Hadamard difference set. 

\subsection{Cayley graphs}
Let $G$ be a group and $S$ be an identity-free subset such that $S^{-1} = S$. The \emph{Cayley graph} on the group $G$ with \emph{connection set} $S$, denoted by $\operatorname{Cay}(G,S)$, is the graph whose vertex set is $G$ and two vertices $x,y$ are adjacent if and only if $xy^{-1}$ belongs to $S$. 

A characterisation of  divisible design graphs that are Cayley graphs was given in \cite{KS21}.

\section{New construction of divisible design graphs}\label{sec:construction}

The construction of the distance-regular graphs from Theorem \ref{thm:SymplecticCovers1} can be modified to produce divisible design graphs.

\begin{theorem}\label{thm:DDG}
Let $q$ be a prime power. 
Let $V$ be a vector space of arbitrary even positive dimension $2t$ over $\mathbb{F}_q$ provided with a nondegenerate symplectic form $B$. Let $A$ be a subgroup of index $r > 1$ in the additive group of $\mathbb{F}_q^+$. 
Let $D = \{c_1+A,\ldots,c_K+A\}$ be a zero-free reversible $(r,K,\Lambda)$ difference set in the (elementary abelian) quotient group $\mathbb{F}_q^+/A$ and let 
$\widetilde{D} = (c_1+A) \cup \ldots \cup (c_K+A)$.
Then the graph $\Delta(2t,q,A,r,D)$ with the vertex set $\{(\alpha+A,x) : \alpha \in \mathbb{F}_q, ~ x \in V\}$ where $(\alpha+A,x) \sim (\beta+A,y)$ if and only if $B(x,y) \in \alpha - \beta + \widetilde{D}$ is a divisible design graph with parameters $(v,k,\lambda_1,\lambda_2,m,n)$, where
\[  v = q^{2t}\cdot r,\quad  k = q^{2t}\cdot K, \quad \lambda_1 = q^{2t} \cdot \Lambda, \quad \lambda_2 = q^{2t}\cdot K^2 / r, \quad m = q^{2t}, \quad n = r.  \]
\end{theorem}
\begin{proof}
First, note that since $D$ is a reversible difference set, we have $-\widetilde{D} = \widetilde{D}$ and so the graph $\Delta(2t,q,A,r,D)$ is undirected. Similarly, since $D$ is zero-free, the graph  $\Delta(2t,q,A,r,D)$ has no loops. 

\medskip \noindent Clearly, the number of vertices of the graph is  $q^{2t} \cdot r$.  

\medskip \noindent We fix a vertex $(\alpha + A,x)$ and count its neighbours $(\beta + A,y)$. We have $(\alpha+A,x) \sim (\beta + A,y)$ if and only if $B(x,y) \in \alpha-\beta + \widetilde{D}$. Choosing $y \in V$ arbitrary ($q^{2t}$ choices), there are $|\widetilde{D}| = K \cdot |A|$ possibilities for $\beta$ such that the latter condition is satisfied. The $q^{2t} K \cdot |A|$ pairs $(\beta,y)$ will give rise to $k=q^{2t} K$ neighbours $(\beta+A,y)$ of $(\alpha+A,x)$.

\medskip \noindent Consider two distinct vertices of the form $(\alpha+A,x)$ and $(\beta+A,x)$. Then $(\beta-\alpha) + A \not= A$. We count the number of common neighbours $(\gamma+A,z)$ of these vertices. For such a vertex to be a common neighbour, we need that
\begin{eqnarray}
B(x,z) & \in & \alpha - \gamma + c_i +A, \label{eq1} \\
B(x,z) & \in & \beta - \gamma + c_j + A, \label{eq2}
\end{eqnarray}
for some $i,j \in \{ 1,2,\ldots,K \}$. These equations imply that $(c_i+A)-(c_j+A) = (\beta-\alpha) + A \not= A$, and so we know that there are $\Lambda$ possibilities for $(c_i,c_j)$. For each of these $\Lambda$ possibilities for $(c_i,c_j)$, (1) implies (2). Taking $z \in V$ arbitrary, we see that $(1)$ holds if $\gamma+A = -B(x,z) + \alpha + c_i +A$. All together, there are thus $\Lambda \cdot |V| = q^{2t} \Lambda$ common neighbours $(\gamma+A,z)$ of the two mentioned vertices.

\medskip \noindent Consider two (distinct) vertices of the form $(\alpha+A,x)$ and $(\beta + A,y)$, where $x \not= y$. We count the number of common neighbours $(\gamma+A,z)$ of these vertices. For such a vertex to be a common neighbour, we need
\begin{eqnarray}
B(x,z) & \in & \alpha - \gamma + c_i + A, \label{eq3} \\ 
B(y,z) & \in & \beta - \gamma + c_j + A, \label{eq4}
\end{eqnarray}
for some $i,j \in \{ 1,2,\ldots,K \}$. These equations imply that 
\begin{eqnarray}
B(x-y,z) & \in & \alpha -\beta+c_i-c_j+A. \label{eq5} 
\end{eqnarray}
If (\ref{eq5}) holds, then (\ref{eq3}) implies (\ref{eq4}). Now, take $i,j \in \{ 1,2,\ldots,K \}$ be arbitrary ($K^2$ possibilities). As $x-y \not= 0$, there are $|A| \cdot q^{2t-1} = \frac{q^{2t}}{r}$ solutions for $z$ that satisfy (\ref{eq5}), and for each such $z$, (\ref{eq3}) holds if we take $\gamma+A = -B(x,z) + \alpha + c_i + A$. All together, this results in $K^2 \frac{q^{2t}}{r}$ common neighbours $(\gamma+A,z)$ of $(\alpha+A,x)$ and $(\beta+A,y)$ (note that any such common neighbour uniquely determines $i$ and $j$ by (\ref{eq3}) and (\ref{eq4})).
  
\medskip \noindent It is now clear that $\Delta(2t,q,A,r,D)$ is a DDG for the mentioned parameters $(v,k,\lambda_1,\lambda_2)$, with two vertices $(\alpha+A,x)$ and $(\beta+A,y)$ belonging to the same class whenever $x=y$. There are thus $m=q^{2t}$ classes, each containing $n=r$ vertices.
\end{proof}

\begin{remark} 
{\em Many examples of DDG's can be obtained as in Theorem \ref{thm:DDG}. Let $q$ be a power of $2$. Then the difference set $D$ from Theorem \ref{thm:DDG} is always reversible and is either trivial or Hadamard. The connection sets of binary hyperbolic and elliptic affine polar graphs \cite[Section 3.3]{BV22}, as well as the connection sets of the collinearity graphs of certain Desarguesian nets are examples of the required Hadamard difference set $D$. Moreover, the support of a Boolean function is a Hadamard difference set if and only if this Boolean function is a bent function \cite[p.185]{BV22}. The total number of bent functions with $2, 4, 6, 8, 10$ variables is, respectively, 2, 8, 896, 5425430528, 99270589265934370305785861242880 (according to OEIS (A004491). Non-equivalent bent functions can give equivalent difference sets though.}    
\end{remark}

\begin{remark}
{\em Let $q$ be a power of $2$. Let $D$ be a difference set in $\mathbb{F}_q^+/A$ containing $A$. Then for any element $y \in (\mathbb{F}_q^+/A) \setminus D$, the set $y+D$ is a zero-free difference set, so the requirement for the difference set $D$ to be zero-free is not so restrictive in this case.}
\end{remark}

\begin{remark}
{\em Let $q$ be an odd prime power. The only reversible zero-free difference set we know in this case is the trivial difference set $D = (\mathbb{F}_q^+/A) \setminus \{A\}$. The resulting divisible design graph is then just the complement of the distance-regular graph from Theorem \ref{thm:SymplecticCovers1}, which is not new in view of \cite[Proposition 4.13]{HKM11} (see Corollary \ref{co2.2}).}      
\end{remark}

\section{Isomorphism with Cayley graphs}\label{sec:SomeGraphsAreCayley}

In this section we show that, for every $q$ that is a power of 2, the distance-regular graphs $\Gamma(2t,q,A,r)$ from Theorem \ref{thm:SymplecticCovers1} and the divisible design graphs $\Delta(2t,q,A,r,D)$ from Theorem \ref{thm:DDG} are Cayley graphs over the elementary abelian 2-group of order $rq^{2t}$.  

We continue with the notation introduced in Theorems \ref{thm:SymplecticCovers1} and \ref{thm:DDG}, but we suppose here that $q$ is even. Also, we put $D = \{ A \}$ and $\widetilde{D}=A$ in the situation of Theorem \ref{thm:SymplecticCovers1}. We denote by $Q$ a quadratic form on $V$ for which the associated alternating bilinear form coincides with $B$, i.e. $B(x,y) = Q(x+y) + Q(x) + Q(y)$ for vectors $x,y \in V$.

Let $G$ denote the quotient group $\mathbb{F}_q^+/A$. Then $G \times V$ is an elementary abelian 2-group. Let $\Omega(2t,q,A,r,D)$ be the Cayley graph defined over the group $G \times V$, with connection set
\[ S = \{ (\alpha+A,x) \in (G \times V) \setminus \{ (A,0) \} \, | \, \alpha + Q(x) + A \in D \}. \]
Note that $-S=S$ as $q$ is even.

\begin{theorem} \label{thm:IsomorphismWithCayleyGraphs}
\begin{enumerate}
\item[$(1)$] If $D = \{A\}$, then the graph $\Gamma(2t,q,A,r)$ is isomorphic to $\Omega(2t,q,A,r,D)$.
\item[$(2)$] If $D$ is a zero-free difference set in $\mathbb{F}_q^+/A$, then the graph $\Delta(2t,q,A,r,D)$ is isomorphic to $\Omega(2t,q,A,r,D)$.
\end{enumerate}
\end{theorem}
\begin{proof}
Let $\phi$ be the bijection from the vertex set of $\Gamma(2t,q,A,r)$ (respectively, $\Delta(2t,q,A,r,D)$) to the vertex set of $\Omega(2t,q,A,r,D)$ defined by the rule
\[ \phi((\alpha+A,x)) = (\alpha+Q(x)+A,x). \]
Then two distinct vertices $(\alpha+A,x)$ and $(\beta+A,y)$ are adjacent in $\Gamma(2t,q,A,R)$ (resp. $\Delta(2t,q,A,r,D)$) whenever $B(x,y) \in \alpha + \beta + \widetilde{D}$, or equivalently $B(x,y)+\alpha+\beta+A \in D$.

On the other hand, the vertices $\phi(\alpha+A,x) = (\alpha + Q(x) + A,x)$ and $\phi(\beta + A,y) = (\beta + Q(y) + A,y)$ are adjacent in $\Omega(2t,q,A,r,D)$ if and only if $(\alpha+\beta + Q(x) + Q(y) + A,x+y) \in S$, i.e. if and only if
\[ \alpha + \beta + Q(x) + Q(y) + Q(x+y) + A = \alpha + \beta + B(x,y) + A \in D. \]
This proves the claims.
\end{proof}

\begin{remark}
{\em Suppose that $r = q$, $A=\{ 0 \}$ and $D = \{ A \}$ (special case of Theorem \ref{thm:IsomorphismWithCayleyGraphs}(1)). Then we can identify $G$ with $\mathbb{F}_q$, $G \times V$ with $\mathbb{F}_q^{2t+1}$ and for $Q$, we can take the quadratic form $x_1x_2+\cdots+x_{2t-1}x_{2t}$ with respect to some basis of $V$. The connection set $S$ then consists of all $(x_0,x_1,\ldots,x_{2t})$ for which $x_0 + x_1x_2 + \cdots + x_{2t-1} x_{2t}=0$. In view of the automorphism  $(x_0,x_1,\ldots,x_{2t}) \mapsto (x_0^2,x_1,\ldots,x_{2t})$ of $G \times V$, the distance-regular graph $\Omega(2t,q,A,r,D)$ can be viewed as a ``parabolic affine polar graph'', which is not strongly regular while the hyperbolic affine polar graphs and elliptic affine polar graphs are known to be strongly regular \cite[Section 3.3]{BV22}.} 
\end{remark}

\begin{remark}
{\em Note that the distance-regular graphs from Theorem \ref{thm:SymplecticCovers1} on $27$ and $125$ vertices are Cayley graphs over non-abelian groups but are not Cayley graphs over the elementary abelian groups of order $27$ and $125$, respectively. So Theorem \ref{thm:IsomorphismWithCayleyGraphs} does not directly extend to the case of odd characteristic.}
\end{remark}

\begin{remark}
{\em
In this remark, we would like to give some comments on the context to which this paper belongs and announce some future results involving two coauthors of the present paper (Bart De Bruyn and Sergey Goryainov). First, we plan to extend the results of this paper by replacing the difference set in Theorem \ref{thm:DDG} with a divisible difference set. In particular, we can announce infinitely many new divisible design graphs that are Cayley graphs and a new recursive construction of divisible difference sets; these results are based on some observations made after submission of this paper. Second, we note that difference sets (more generally, divisible difference sets) can be plugged into another construction of antipodal distance-regular graphs (see \cite[Proposition 12.5.3]{BCN89}); together with the idea of increasing the dimension of the vector space, this would serve as a wide generalisation of a result from \cite{MR26}, which was also independently obtained by Mikhail Muzychuk \& Grigory Ryabov and Bart De Bruyn, Sergey Goryainov \& Weihao Yan. We thus announce infinitely many new divisible design graphs.
Third, we announce that difference sets can be similarly plugged into three more infinite families of antipodal distance-regular graphs (see \cite{C91} and \cite[Theorem 1]{T15}; also, see \cite{T22}).
}    
\end{remark}

\section*{Acknowledgments}

\noindent Bart De Bruyn and Sergey Goryainov are supported by the Natural Science Foundation of Hebei Province (A2023205045) and the 111 Center (Grant No.D26018). Sergey Goryainov also thanks the Special Research Fund of Ghent University (bof/baf/4y/2024/01/354) for supporting his visits to Ghent University in September 2025 and January-February 2026.


\begin{thebibliography}{99}

\bibitem{BJL99} T. Beth, D. Jungnickel, H. Lenz, \emph{Design Theory. Vol I}, Second edition, Encyclopedia of Mathematics and its Applications 69, Cambridge University Press, Cambridge, 1999.

\bibitem{BDG25}
A. Bhowmik, B. De Bruyn, S. Goryainov, \emph{Divisible design graphs with selfloops}, Discrete Math. 349 (2026), no. 3, Paper No. 114824, 15 pp. \url{https://doi.org/10.1016/j.disc.2025.114824}

\bibitem{BG24} A. Bhowmik, S. Goryainov, \emph{Divisible design graphs from symplectic graphs over rings with precisely three ideals}, arXiv:2412.04962, \url{https://arxiv.org/abs/2412.04962}

\bibitem{B77} R. C. Bose, \emph{Symmetric group divisible designs with the dual property},  J. Statist. Plann. Inference 1 (1977), no. 1, 87--101. \url{https://doi.org/10.1016/0378-3758(77)90008-8}

\bibitem{BCN89} A. E. Brouwer,  A. M. Cohen, A. Neumaier, \emph{Distance-regular  graphs}, Springer-Verlag, Berlin, 1989.

\bibitem{BV22} A. E. Brouwer, H. Van Maldeghem, \emph{Strongly regular graphs}, Encyclopedia of Mathematics and its Applications 182, Cambridge University Press, Cambridge, 2022. 

\bibitem{C91}
P.J. Cameron, \emph{Covers of graphs and EGQs}, Discrete Math. 97(1–3) (1991) 83--92.
\url{https://doi.org/10.1016/0012-365X(91)90424-Z}

\bibitem{CH14} D. Crnkovi\'c, W. H. Haemers, \emph{Walk-regular divisible design graphs}, Des. Codes Cryptogr. 72 (2014), no. 1, 165--175. \url{https://doi.org/10.1007/s10623-013-9861-0}

\bibitem{CS22} D. Crnkovi\'c, A. \v{S}vob, \emph{New constructions of divisible design Cayley graphs}, Graphs Combin. 38 (2022), no. 1, Paper No. 17, 8 pp. \url{https://doi.org/10.1007/s00373-021-02440-4}

\bibitem{DGHS24} B. De Bruyn, S. Goryainov, W. H. Haemers, L. Shalaginov, \emph{Divisible design graphs from the symplectic graph},  Des. Codes Cryptogr. 93 (2025), no. 5, 1401--1424. \url{https://doi.org/10.1007/s10623-024-01557-w}

\bibitem{GK24} A. L. Gavrilyuk, V. V. Kabanov,  \emph{Strongly regular graphs decomposable into a divisible design graph and a Hoffman coclique},  Des. Codes Cryptogr. 92 (2024), no. 5, 1379--1391. \url{https://doi.org/10.1007/s10623-023-01348-9}

\bibitem{GK25} A. L. Gavrilyuk, V. V. Kabanov, \emph{Strongly regular graphs decomposable into a divisible design graph and a Delsarte clique}, Des. Codes Cryptogr. 93 (2025), no. 6, 2177--2189. \url{https://doi.org/10.1007/s10623-024-01563-y}

\bibitem{GHKS19} S. Goryainov, W. H. Haemers, V. V. Kabanov, L. Shalaginov, \emph{Deza graphs with parameters $(n, k, k-1, a)$ and $\beta = 1$},  J. Combin. Des. 27 (2019), no. 3, 188--202. \url{https://doi.org/10.1002/jcd.21644}

\bibitem{GHKL25} S. Goryainov, W. H. Haemers, E. V. Konstantinova, H. Li, \emph{Thin divisible designs graphs: an interplay between fixed-point free involutions of $(v,k,\lambda)$-graphs and symmetric weighing matrices}, arXiv:2512.16653, \url{https://arxiv.org/abs/2512.16653}

\bibitem{HKM11} W. H. Haemers, H. Kharaghani, M. A. Meulenberg, \emph{Divisible design graphs}, J. Combin. Theory Ser. A 118 (2011), no. 3, 978--992.  \url{https://doi.org/10.1016/j.jcta.2010.10.003}

\bibitem{K22} V. V. Kabanov, \emph{New versions of the Wallis-Fon-Der-Flaass construction to create divisible design graphs}, Discrete Math. 345 (2022), no. 11, Paper No. 113054, 9 pp. \url{https://doi.org/10.1016/j.disc.2022.113054}

\bibitem{K23} V. V. Kabanov, \emph{A new construction of strongly regular graphs with parameters of the complement symplectic graph}, Electron. J. Combin. 30 (2023), no. 1, Paper No. 1.25, 9 pp. \url{https://doi.org/10.37236/11343}

\bibitem{K26} V. V. Kabanov, \emph{Construction of divisible design graphs using affine designs}, Discrete Math. 349 (2026), no. 2, Paper No. 114717, 7 pp. \url{https://doi.org/10.1016/j.disc.2025.114717}

\bibitem{KS21} V. V. Kabanov, L. Shalaginov, \emph{On divisible design Cayley graphs}, Art Discrete Appl. Math. 4 (2021), no. 2, Paper No. 2.02, 9 pp. \url{https://doi.org/10.26493/2590-9770.1340.364}

\bibitem{M08} M. A. Meulenberg, \emph{Divisible design graphs}, Master's thesis, Tilburg University (2008). \url{http://alg.imm.uran.ru/dezagraphs/Divisible_design_graphs_M.A._Meulenberg.pdf}

\bibitem{MR26} M. Muzychuk, G. Ryabov, \emph{Directed strongly regular graphs and divisible design graphs from Tatra association schemes}, arXiv:2601.09955, \url{https://arxiv.org/abs/2601.09955}

\bibitem{P22} D. Panasenko, \emph{The vertex connectivity of some classes of divisible design graphs},  Sib. Elektron. Mat. Izv. 19 (2022), no. 2, 426--438. \url{http://semr.math.nsc.ru/v19/n2/p426-438.pdf}

\bibitem{PS22} D. Panasenko, L. Shalaginov, \emph{Classification of divisible design graphs with at most $39$ vertices},  J. Combin. Des. 30 (2022), no. 4, 205--219. \url{https://doi.org/10.1002/jcd.21818}

\bibitem{R26} G. Ryabov, \emph{Divisible design graphs from Higmanian association schemes}, arXiv:2601.18370, \url{https://arxiv.org/abs/2601.18370}

\bibitem{S21} L. Shalaginov, \emph{Divisible design graphs with parameters $(4n, n+2, n-2, 2, 4, n)$ and $(4n, 3n-2, 3n-6, 2n-2, 4, n)$},  Sib. Elektron. Mat. Izv. 18 (2021), no. 2, 1742--1756. \url{https://doi.org/10.33048/semi.2021.18.134}

\bibitem{T15}
L. Y. Tsiovkina, \emph{Two new infinite families of arc-transitive antipodal distance-regular graphs of diameter three with related to groups $Sz(q)$ and $^2G_2(q)$}, Journal of Algebraic Combinatorics 41, 1079--1087 (2015).

\bibitem{T22} L. Tsiovkina, \emph{Covers of complete graphs and related association schemes}, J. Combin. Theory Ser. A 191 (2022), Paper No. 105646, 33 pp. \url{https://doi.org/10.1016/j.jcta.2022.105646}

\end{thebibliography}
\end{document}